\documentclass[11pt,a4paper]{article}
\usepackage{amsmath, amssymb, amscd, url, rotating}

\newtheorem{theorem}{Theorem}[section]
\newtheorem{lemma}[theorem]{Lemma}

\newenvironment{proof}[1][Proof:]{\begin{trivlist}
\item[\hskip \labelsep {\bfseries #1}]}{\end{trivlist}}
\newenvironment{definition}[1][Definition:]{\begin{trivlist}
\item[\hskip \labelsep {\bfseries #1}]}{\end{trivlist}}
\newenvironment{example}[1][Example:]{\begin{trivlist}
\item[\hskip \labelsep {\bfseries #1}]}{\end{trivlist}}
\newenvironment{remark}[1][Remark:]{\begin{trivlist}
\item[\hskip \labelsep {\bfseries #1}]}{\end{trivlist}}

\newcount\colveccount
\newcommand*\colvec[1]{
        \global\colveccount#1
        \begin{pmatrix}
        \colvecnext
}
\def\colvecnext#1{
        #1
        \global\advance\colveccount-1
        \ifnum\colveccount>0
                \\
                \expandafter\colvecnext
        \else
                \end{pmatrix}
        \fi
}
\newcommand{\qed}{\nobreak \ifvmode \relax \else
      \ifdim\lastskip<1.5em \hskip-\lastskip
      \hskip1.5em plus0em minus0.5em \fi \nobreak
      \vrule height0.75em width0.5em depth0.25em\fi}
\begin{document}

\title{A function to calculate all relative prime numbers up to the product of the first n primes}
\author{Matthias Schmitt\footnote{mschmitt@db12.de, diploma mathematician graduated in Mathematical Institute, University of Cologne, 1997}}
  
\maketitle

\begin{abstract}
We prove an isomorphism between the finite domain $F_{\prod{p_i}}$ and the new defined set of prime modular numbers. This definition provides some insights about relative prime numbers. We provide an inverse function from the prime modular numbers into $F_{\prod{p_i}}$. With this function we can calculate all numbers from 1 up to the product of the first n primes that are not divisible by the first n primes.  This function provides a non sequential way for the calculation of prime numbers.
\end{abstract}

\section{prime modular numbers}

\begin{definition}
Let $n \in N$  and $p_i, i=1..n$ the first n prime numbers.\newline
The \textbf{prime modular numbers of level n} are tuples of length n defined by
\begin{center}

$PM(n) := \{\colvec{6}{t_1 \mod 2}{t_2 \mod 3}{t_3 \mod 5}{t_4 \mod 7}{...}{t_n \mod p_n}| t_i \in N\}$.
\end{center}
The set of all prime modular numbers of length n is called $PM(n)$.
\end{definition}

\begin{example}
For level n=1 there are two prime modular numbers:
$$PM(1) = \{\colvec{1}{0},\colvec{1}{1}\}$$
For level n=2 there are six prime modular numbers:
$$PM(2) = \{\colvec{2}{0}{0},\colvec{2}{0}{1},\colvec{2}{0}{2},\colvec{2}{1}{0},\colvec{2}{1}{1},\colvec{2}{1}{2}\}$$
\end{example}

\begin{definition}
The \textbf{primorial number} $\prod\limits_{i=1}^np_i$ is defined as the product of the first n prime numbers. 
\end{definition}

\begin{lemma}
The number of elements in PM(n) equals the primorial number $\prod\limits_{i=1}^np_i$.
\end{lemma}

\begin{proof}
The number of different element on position $i\in \{1,..,n\}$ equals $p_i$ due to the defintion of the mod() function. The number of permutations of the first n primes equals  $\prod\limits_{i=1}^np_i$. \qed
\end{proof}

\section{Isomorphism $f: F_{\prod{p_i}} \rightarrow PM(n)$}

\begin{definition}
We define addition in PM(n) by:
\begin{center}
$\colvec{5}{s_1 \mod 2}{s_2 \mod 3}{s_3 \mod 5}{...}{s_n \mod p_n} 
+ \colvec{5}{t_1 \mod 2}{t_2 \mod 3}{t_3 \mod 5}{...}{t_n \mod p_n} 
:= \colvec{5}{(s_1 + t_1) \mod 2}{(s_2 + t_2) \mod 3}{(s_3 + t_3) \mod 5}{...}{(s_n + t_n) \mod p_n}$.
\end{center}
\end{definition}

\begin{definition}
We define multiplication in PM(n) by:
\begin{center}
$\colvec{5}{s_1 \mod 2}{s_2 \mod 3}{s_3 \mod 5}{...}{s_n \mod p_n} 
* \colvec{5}{t_1 \mod 2}{t_2 \mod 3}{t_3 \mod 5}{...}{t_n \mod p_n} 
:= \colvec{5}{(s_1 * t_1) \mod 2}{(s_2 * t_2) \mod 3}{(s_3 * t_3) \mod 5}{...}{(s_n * t_n) \mod p_n}$.
\end{center}
\end{definition}

\begin{definition}
Let be $F_{\prod{p_i}} = \{1,2,..,\prod\limits_{i=1}^np_i\}$. \newline
$F_{\prod{p_i}}$ is a finite integral domain of size $\prod\limits_{i=1}^np_i$. \newline
We define the homomorphism $f: F_{\prod{p_i}} \rightarrow PM(n)$:
$$f(k) := \colvec{5}{k \mod 2}{k \mod 3}{k \mod 5}{...}{k \mod p_n}$$
\end{definition}

\newtheorem{propIso}{Proposition}
\begin{propIso}[Isomorphism from $F_{\prod{p_i}}$ into PM(n)]
Let be $ k \in F_{\prod{p_i}}$ and f the homomorphism: $$f(k) :=\colvec{5}{k \mod 2}{k \mod 3}{k \mod 5}{...}{k \mod p_n} \in PM(n)$$
f is an isomorphism. 
\end{propIso}

\begin{proof}
To proof the isomorphism, we show that 
$f(m+n) = f(m) + f(n)$,  $f(m*n)  = f(m) * f(n)$ and f is injective.

a) $$f(m+n) = \colvec{5}{m+n \mod 2}{m+n \mod 3}{m+n \mod 5}{...}{m+n \mod p_n}
    = \colvec{5}{m \mod 2 + n \mod 2}{m \mod 3 +n \mod 3}{m \mod 5 + n \mod 5}{...}{m \mod p_n + n \mod p_n}  $$ $$  = \colvec{5}{m \mod 2}{m \mod 3}{m \mod 5}{...}{m \mod p_n} + \colvec{5}{n \mod 2}{n \mod 3}{n \mod 5}{...}{n \mod p_n} = f(m) + f(n)$$

b) $$f(m*n) = \colvec{5}{m*n \mod 2}{m*n \mod 3}{m*n \mod 5}{...}{m*n \mod p_n}
    = \colvec{5}{m \mod 2 * n \mod 2}{m \mod 3 *n \mod 3}{m \mod 5 * n \mod 5}{...}{m \mod p_n * n \mod p_n}  $$ $$  = \colvec{5}{m \mod 2}{m \mod 3}{m \mod 5}{...}{m \mod p_n} * \colvec{5}{n \mod 2}{n \mod 3}{n \mod 5}{...}{n \mod p_n} = f(m) * f(n)$$

c) f is injective:\newline
$$f(1)=\colvec{4}{1}{1}{...}{1} \text{ and } f(\prod{p_i}) = \colvec{4}{0}{0}{...}{0}$$
$$\forall m \in F_{\prod{p_i}}: f(m+1) = \colvec{4}{m+1 \mod 2}{m+1 \mod 3}{...}{m+1 \mod p_n} $$
Thus $$\forall m,n \in F_{\prod{p_i}}: f(m) = f(n) => f(m+1) = f(n+1)$$
Thus $$\forall m,n,r \in F_{\prod{p_i}}: f(m) = f(n) => f(m+r) = f(n+r) \text{  }[\alpha]$$
Suppose there are two values $$ m,n \in F_{\prod{p_i}} \text{ and }  n>m\text{ with }  f(m) = f(n)$$.
Then there is $$ r:= \prod{p_i} -n \text{ and } f(n+r)= f(\prod{p_i})= \colvec{4}{0}{0}{...}{0}$$
Due to $[\alpha]$: $$f(n+r) = f(m+r) = \colvec{4}{0}{0}{...}{0}$$
So there must be an value $m < \prod{p_i}$ that is divisible by the first n primes.
This is impossible, therefore f is injective.\qed

\end{proof}

\begin{example}
For level n=2:
\begin{center}
\begin{tabular}{c|cccccc|}
  k& 1& 2& 3& 4& 5& 6\\
\hline 
  $f(k)$& $\colvec{2}{1}{1}$& $\colvec{2}{0}{2}$& $\colvec{2}{1}{0}$& $\colvec{2}{0}{1}$& $\colvec{2}{1}{2}$& $\colvec{2}{0}{0}$ \\ 
\end{tabular}
\end{center}

For level n=3:
  $$f(1)=\colvec{3}{1}{1}{1}, f(2)=\colvec{3}{0}{2}{2}, f(3)=\colvec{3}{1}{0}{3}, f(4)=\colvec{3}{0}{1}{4}, f(5)=\colvec{3}{1}{2}{0}, f(6)=\colvec{3}{0}{0}{1},$$ 
  $$f(7)=\colvec{3}{1}{1}{2}, f(8)=\colvec{3}{0}{2}{3}, f(9)=\colvec{3}{1}{0}{4}, f(10)=\colvec{3}{0}{1}{0}, f(11)=\colvec{3}{1}{2}{1}, f(12)=\colvec{3}{0}{0}{2},$$ 
  $$f(13)=\colvec{3}{1}{1}{3}, f(14)=\colvec{3}{0}{2}{4}, f(15)=\colvec{3}{1}{0}{0}, f(16)=\colvec{3}{0}{1}{1}, f(17)=\colvec{3}{1}{2}{2}, f(18)=\colvec{3}{0}{0}{3},$$ 
  $$f(19)=\colvec{3}{1}{1}{4}, f(20)=\colvec{3}{0}{2}{0}, f(21)=\colvec{3}{1}{0}{1}, f(22)=\colvec{3}{0}{1}{2}, f(23)=\colvec{3}{1}{2}{3}, f(24)=\colvec{3}{0}{0}{4},$$ 
  $$f(25)=\colvec{3}{1}{1}{0}, f(26)=\colvec{3}{0}{2}{1}, f(27)=\colvec{3}{1}{0}{2}, f(28)=\colvec{3}{0}{1}{3}, f(29)=\colvec{3}{1}{2}{4}, f(30)=\colvec{3}{0}{0}{0} $$

\end{example}

\newtheorem{corIso}{Corollary}
\begin{corIso}
 PM(n) is a finite integral domain of size $\prod\limits_{i=1}^np_i$.
\end{corIso}

\begin{remark}
The main drawback for calculation with prime modular numbers, is that it is no longer possible to compare two numbers and decide which one is bigger. But we get some other advantages instead, because it is much easier to decide whether a number is relative prime.

There is an inverse function $f^{-1}$:  $PM(n) \rightarrow F_{\prod{p_i}}$ that will be discussed in chapter 4.

\end{remark}

\section{relative prime elements}

\begin{definition}
Given the set PM(n). An element in PM(n) is said to be \textbf{relative prime} if it has no zeros in its tuple. 
\end{definition}

\begin{lemma}
The size of the set PM(n) is $\prod\limits_{i=1..n} (p_i -1)$.
\end{lemma}

\begin{proof}
The number of values in a prime modular number on each position without 0 is $p_i -1$.
Therefore the number of possible permutations is $\prod\limits_{i=1..n} (p_i -1)$. \qed
\end{proof}

\begin{example}
in PM(3) there are 8 relative prime elements: \newline $ 1=\colvec{3}{1}{1}{1}$,$7=\colvec{3}{1}{1}{2}$,$13=\colvec{3}{1}{1}{3}$,$19=\colvec{3}{1}{1}{4}$,\newline $11=\colvec{3}{1}{2}{1}$,$17=\colvec{3}{1}{2}{2}$,$23=\colvec{3}{1}{2}{3}$,$29=\colvec{3}{1}{2}{4}$.
\end{example}

\begin{definition}
Given the set PM(n) and the isomorphism $f: F_{\prod{p_i}} \rightarrow PM(n)$. We define $f^{-1}: PM(n) \rightarrow F_{\prod{p_i}}$ as the \textbf{inverse function} to f with $$\forall k \in F_{\prod{p_i}} : f^{-1}(f(k)) = k$$. 
\end{definition}

\begin{definition}
Given the set PM(n). We define the \textbf{prime} elements \newline $t \in PM(n)$ as all elements, with $ f^{-1}(\colvec{4}{t_1}{t_2}{...}{t_{p_n}}) $ is prime  in $F_{\prod{p_i}}$. 
\end{definition}

\begin{lemma}
Let $p_{n+1}$ be the (n+1)-th prime number in N.
\newline Let $Q = \{q \textup{ is prime } \wedge  p_{n+1} \leq q < \prod\limits_{i=1}^np_i \}$.  \newline $ \forall q \in Q:  f(q) $ is relative prime.
\end{lemma}

\begin{proof}
Each $q \in Q$ is prime and therefore not divisible by the first n primes $p_1 .. p_n$.  
$$f(q) = \colvec{5}{q \mod 2}{q \mod 3}{q \mod 5}{...}{q \mod p_n}$$ has no zero value in its tuple and therefore is relative prime. \qed
\end{proof}

\begin{remark}
$f(1)=\colvec{4}{1}{1}{..}{1}$ is never prime in PM(n). 
\newline
In PM(2) and  PM(3) all relative prime elements (except 1) are prime.
\newline
For $n \geq 4$ there are many other relative prime elements in PM(n), that are not prime. 
\newline
In PM(4) (range 11..210) all examples for this are $$11^2=121=\colvec{4}{1}{1}{1}{2},11*13=143=\colvec{4}{1}{2}{3}{3},$$ $$11*17=187=\colvec{4}{1}{1}{1}{6},13^2=169=\colvec{4}{1}{1}{4}{1},11*19=209=\colvec{4}{1}{2}{4}{6}$$.

\end{remark}

\begin{lemma}
Let $p_{n+1}$ be the (n+1)-th prime number in N.
Each relative prime number $t \in PM(n)$ with $f^{-1}(t) = k$ and  $p_{n+1} \leq k < p_{n+1}^2$ is prime.
\end{lemma}

\begin{proof}
Each relative prime number in PM(n) is not divisible by the first n primes. Therefore the minimal relative prime number q, that is not prime must have  $f^{-1}(q) = p_{n+1}^2$. \qed
\end{proof}

\begin{lemma}
Given the set PM(n). We define the subset $$PM^{'}(n) := \{t \in PM(n) \mid t \text{ is not prime } \wedge \text{ minimal prime factor }(f^{-1}(t)) > p_n\}$$ . $PM^{'}(n)$ contains all relative prime elements that are not prime.
\end{lemma}

\begin{proof}
With minimal prime factor $(f^-1(t)) > p_n =>$ \newline all tuple values are $> 0 =>$ \newline t is relative prime per definitionem. \qed 
\end{proof}

\newtheorem{propPrim}{Proposition}
\begin{propPrim}

Let $p_{n+1}$ be the (n+1)-th prime number in N.
 The number of all primes q with $p_{n+1} \leq q < \prod\limits_{i=1}^np_i$ is
 $$\prod\limits_{i=1..n} (p_i -1) -1 - \text{ size of } PM^{'}(n)$$
\end{propPrim}

\begin{proof}
The number of relative primes in PM(n) is $\prod\limits_{i=1..n} (p_i -1)$.

From this we substract one for the relative prime $1=\colvec{4}{1}{1}{..}{1}$.
The remaining relative primes, that are not prime are defined in the set $PM^{'}(n)$. \qed
\end{proof}

\begin{example}
The number of primes between 6 and 30: (1*2*4) -1 -0 = 7 \newline
The number of primes between 10 and 210: (1*2*4*6) -1 -5  = 42.
\end{example}

\newtheorem{corRel}{Corollary}
\begin{corRel}
 Let $p_{n+1}$ be the (n+1)-th prime number in N.
 The number of all primes q with $q < \prod\limits_{i=1}^np_i$ is
 $$ n + \prod\limits_{i=1..n} (p_i -1) -1 - \text{ size of } PM^{'}(n)$$
\end{corRel}

\section{the inverse function $f^{-1}$}

\begin{definition}
Let the set PM(n) be given. We define the \textbf{unary} elements in PM(n) as all elements with (n-1) zeros and one 1 in the tuple. 
\end{definition}

\begin{example}
There are n unary tuples in PM(n).
In PM(4) there are 4 unary elements $\colvec{4}{1}{0}{0}{0}$,$\colvec{4}{0}{1}{0}{0}$,$\colvec{4}{0}{0}{1}{0}$,$\colvec{4}{0}{0}{0}{1}$.
\end{example}

\begin{lemma}
Each unary element in PM(n) generates a subset of PM(n) that is a finite field.
Therefore each element t in this subset has an inverse element.
\end{lemma}

\begin{proof}
Let $$PM(n)_{u_i} := \{ \colvec{6}{0}{..}{k}{..}{0}{0} | k \in \{0,..,p_i -1\}\}$$ be the subset of all tuples, where all tuple values are zero, but on the i-th position. 
This subset forms a field with addition and multiplication and is isomorph to the finite field $Z/Z_{p_i}$. \qed
\end{proof}

\begin{lemma}
Each element in PM(n) can be constructed by the unary elements.
\end{lemma}

\begin{proof}
$$\colvec{5}{k_1 \mod 2}{k_2 \mod 3}{k_3 \mod 5}{...}{k_n \mod p_n} =
\sum_{i = 1}^{k_1} \colvec{5}{1}{0}{0}{...}{0} +
\sum_{i = 1}^{k_2} \colvec{5}{0}{1}{0}{...}{0} +
... +
\sum_{i = 1}^{k_n} \colvec{5}{0}{0}{0}{...}{1} $$

\qed
\end{proof}

\begin{definition}
Let the finite field $Z/Z_{p_i} $ with $p_i$ prime  be given.\newline
For each  element $t \in Z/Z_{p_i} $ \newline there exists an inverse element $t^{-1}$ with $t * t^{-1} = 1$.
\newline
We define $t \textbf{  modInverse}(p_i) := t^{-1}$
\end{definition}

\newtheorem{propUnary}{Proposition}
\begin{propUnary}
For an unary element $u_i \in PM(n): v_i := f^{-1}(u_i)$ can be calculated by:
  $$v_i =  {(\prod\limits_{j=1}^np_j) / p_i} * [(\prod\limits_{j=1}^np_j) / p_i \text{ modInverse }(p_i)]$$
\end{propUnary}

\begin{proof}
The value $v_i = f^{-1}(u_i)$ must be \newline [condition a]: divisible by all primes $ p_j, j \in \{1..n\}, j \neq i  $ except $p_i$ \newline
and [condition b]: $v_i \mod(p_i) = 1 $.\newline
$v_i := (\prod\limits_{j=1}^np_j / p_i)$ resolves condition a , but not in all cases condition b.

To resolve [condition b] we multiple the value $v_i$ by $(\prod\limits_{j=1}^np_j / p_i) \text{ modInverse }(p_i)$.
\newline
$v_i^* := {v_i} * [(\prod\limits_{j=1}^np_j) / p_i \text{ modInverse }(p_i)] \mod p_i = 1$ resolves both conditions.
\newline
$v_i$ is an element in the field $PM(n)_{u_i}$. Therefore there exists an inverse element.
\newline
Due to the isomorphism in chapter 2 there can only be one $v_i^* \text{ with } f(v_i^´)=u_i$. \qed
\end{proof}

\newtheorem{thmPM}{Theorem}
\begin{thmPM}[inverse function]
Let PM(n) be given.\newline Let g: PM(n)  $ \rightarrow F_{\prod{p_i}}: $ 

Let the map be defined by $$g(\colvec{5}{t_1}{t_2}{t_3}{...}{t_n}) = (\sum_{i=1..n} t_i * v_i) \mod (\prod\limits_{i=1}^np_i)$$ \newline Then g is the inverse function to f.
\end{thmPM}

\begin{proof}
Each $k \in PM(n)$ can be calculated by its unary elements. 
For each unary element the value $v_i$ is known. 
\qed

\end{proof}

\begin{remark}
For all $t_i = 0 \rightarrow g(t) = 0$. In $F_{\prod{p_i}}$ this is equivalent to $\prod{p_i}$. 
\end{remark}

\begin{example}

\begin{center}
\begin{tabular}{c|c|c|c|c|}
  n& range& rel. primes& coeffs\\
\hline 
  1& 1 - 2& 1& $f^{-1}(t)=(1t_1) \mod 2$\\ 
  2& 1 - 6& 2& $f^{-1}(t)=(3t_1 + 4t_2) \mod 6$\\ 
  3& 1 - 30& 8& $f^{-1}(t)=(15t_1 + 10t_2 + 6t_3) \mod 30 $\\ 
  4& 1 - 210& 48& $f^{-1}(t)=(105t_1 + 70t_2 + 126t_3 + 120t_4) \mod 210 $\\ 
  5& 1 - 2310& 480& $f^{-1}(t)=(1155t_1 + 1540t_2 + 1386t_3 + 330t_4+ 210t_5) \mod 2310 $\\ 
\end{tabular}

with $t_1 \in \{0,1\}, t_2 \in \{0,1,2\}, t_3 \in \{0,1,2,3,4\}, t_4 \in \{0,1,2,3,4,5,6\},t_5 \in \{0,1,2,3,4,5,6,7,8,9,10\}$
\end{center}
\end{example}

\section{conclusions}
The inverse function provides an algorithm to test for higher primes:

We can iterate all relative prime elements in PM(n) to calculate all relative prime values up to $\prod\limits_{i=1}^np_i$.

Start for example with $\colvec{4}{1}{2}{4}{6}$ and subtract one from the last value $\colvec{4}{1}{2}{4}{5}, \colvec{4}{1}{2}{4}{4}, ...$ until $\colvec{4}{1}{2}{4}{1}$ and then switch to $\colvec{4}{1}{2}{3}{6}$ and then repeat until we reach $\colvec{4}{1}{1}{1}{1}$.
For each element we can calculate $f^{-1}(t)$ as a prime candidate. This is really similar to the wheel factorization of primes, but not in sequential order. Instead of iterating the wheel multiple times we just go up to $\prod\limits_{i=1}^np_i$.

There are some questions open to make this algorithm effective. We still have to check each relative prime element for its primeness.
As part of a solution, it is possible to extend each element $$t = \colvec{4}{t_1}{t_2}{..}{t_n} \in PM(n)$$ for tuple positions $t_i > n$. Together with $f^{-1}(t) <=  \prod\limits_{i=1}^np_i$ these extensions are unique. \newline The extended element $t'= \colvec{7}{t_1}{t_2}{..}{t_n}{t_{n+1}}{..}{t_{n+m}}$ can be easily checked for primeness.

\section{remarks}
Thanks to all the reviewers for their valuable feedback.
A special thanks to Ralf Schiffler (University of Connecticut) for his help with the calculation of the inverse function.
\newline
After publication of this article Ramin Zahedi contacted me with his own article, which provides similar results from a different perspective.\cite{zahedi} 

\end{document}